\documentclass[11pt]{amsart}

\usepackage{url}
\usepackage{hyperref}

\usepackage{amssymb,amsmath,amscd,graphicx,
latexsym,amsthm}
\usepackage{amssymb,latexsym,amsmath,amscd,graphicx}
\usepackage[mathscr]{eucal}
 \input xy

\usepackage[letterpaper, centering]{geometry}

 \xyoption{all}

\def\Q{{\mathbb Q}}

\def\K{{\mathbb K}}

\def\OO{{\mathcal O}}

\def\LL{\mathcal{L}}

\def\FF{\mathcal{F}}

\def\GG{\mathcal{G}}

\def\II{\mathcal{I}}

\def\LL{\mathcal{L}}
\def\PP{\mathcal{P}}

\def\MM{\mathcal M}

\def\Ad{\widehat{A}}

\def\l{{\underline{\textit{l}}\,}}

\def\m{{\underline{\textit{m}}}}

\def\KA{\mathrm{K}(A)}

\def\KApr{\mathrm{K}(A^{\prime})}

\def\Pic{{Pic}}
\def\Pic0{\mathrm{Pic}^0}

\def\deg{\mathrm{deg}}
\def\Char{\mathrm{char}}

\def\la{{\langle}}
\def\ra{{\rangle}}

\theoremstyle{plain}

\newtheorem{theorem}{Theorem}[section]

\newtheorem{proposition/example}[theorem]{Proposition/Example}
\newtheorem{definition/theorem}[theorem]{Definition/Theorem}
\newtheorem{proposition}[theorem]{Proposition}

\newtheorem{corollary}[theorem]{Corollary}
\newtheorem{lemma}[theorem]{Lemma}

\newtheorem{claim}[theorem]{Claim}

\theoremstyle{definition}
\newtheorem{definition}[theorem]{Definition}

\newtheorem{remark}[theorem]{Remark}

\newtheorem{conjecture/question}[theorem]{Conjecture/Question}

\newtheorem{remark/definition}[theorem]{Remark/Definition}
\newtheorem{notation/assumptions}[theorem]{Assumptions/Notation}

\numberwithin{equation}{section}

 \theoremstyle{remark}

\begin{document}  

\title{Syzygies of Kummer varieties}

 \author{Federico Caucci}
\address{Dipartimento di Matematica ``Federigo Enriques'', Universit\`a degli Studi di Milano,
Via Cesare Saldini 50, 20133 Milano -- Italy}
 \email{federico.caucci@unimi.it }
 \thanks{The author is supported by the ERC Consolidator Grant ERC-2017-CoG-771507-StabCondEn.
}

\maketitle

\setlength{\parskip}{.1 in}

\begin{abstract} 
We study syzygies of  Kummer varieties proving that their behavior is half of the   abelian varieties case. Namely, 
an $m$\emph{-th} power of an ample line bundle  on a Kummer variety
satisfies the Green-Lazarsfeld property $(N_p)$, if $m > \frac{p+2}{2}$.  
 \end{abstract}

\section{Introduction}

Let $A$ be an abelian variety defined over an algebraically closed field $\K$. Given an integer $p \geq 0$ and an ample line bundle $L$ on $A$, Lazarsfeld's conjecture  says that $L^m$ satisfies the property 
$(N_p)$, if
\begin{equation}\label{stimaLaz}
m > p+2.
\end{equation}
 We refer the reader to \S \ref{GLprop} for the definition of  property $(N_p)$. Here, we just content to mention  that $(N_0)$ means $L^m$ is projectively normal, and $(N_1)$ means that, in addition,  the homogeneous ideal of $A$ in $\mathbb{P}(H^0(A, L^m)^{\vee})$ is generated by quadrics.
Lazarsfeld's conjecture was proved by Pareschi \cite{pa1}  in characteristic zero and, more recently, by the author  in general as a consequence of the main result of \cite{ca}.
Note  that \eqref{stimaLaz} does \emph{not} depend on the dimension of $A$, and one has exactly the same behavior as in the case of elliptic curves which is classically due to Green \cite{green}.
 This strongly contrasts with some other situations like, e.g.,
syzygies of ``hyper-adjoint'' line bundles (see \cite{el} and the very recent improvement in \cite{bala}),  Mukai's conjecture on syzygies of adjoint line bundles (see \cite[pp.\ 118-119]{laI}) generalizing Fujita's conjecture, or 
the recent results of Jiang  \cite{zhinew}
giving
  partial answers   to
  Fulton's question on syzygies of rational homogeneous varieties \cite[Problem 4.5]{el}. 
In this paper, we deal with another class of varieties where a similar independent-from-dimension behavior for syzygies arises: Kummer varieties. 

Let us assume that $\Char(\K) \neq 2$ for the remainder of the introduction. 
The \emph{Kummer variety}
$\mathrm{K}(A)$   associated to an abelian variety $A$ over $\K$ is the quotient  of $A$ by the action of the inverse morphism $\iota \colon A \to A$.  The quotient morphism is denoted by $\pi \colon A \to \KA$. Given an ample line bundle $G$ on $\KA$, it is well-known by Mumford 
that the pullback $\pi^* G$ is even, i.e., $\pi^* G = L^2$ for an ample line bundle $L$ on $A$. Therefore, 
one naturally expects  (see \cite[\S 2.4]{tithesis}) that, when $m$ goes like half of right hand side of \eqref{stimaLaz}, the syzygies of $G^m$ are as simple as possible
up to the $p$\emph{-th} step. This is precisely what we prove here.

\begin{theorem}\label{main}
Let $\KA$ be the Kummer variety associated to an abelian variety $A$ defined over an algebraically closed field $\K$, with $\Char(\K) \neq 2$. Let   $G$ be an ample line bundle on $\KA$ and $m, p \geq 0$ be integers. If 
\[
m > \frac{p+2}{2},
\] 
then $G^m$ satisfies the property $(N_p)$.
\end{theorem}

Sasaki \cite{sa1}  showed that $G^m$ is projectively normal, i.e., it has the property $(N_0)$, if $m > 1$.
It was later proved by Khaled \cite{khaled} that the homogeneous ideal of $\KA$ in $\mathbb P(H^0(\KA, G^m)^{\vee})$ is generated by quadrics if $m > 2$, and by quadrics and cubics if $m = 2$. This, in turn,  improved a previous result of Sasaki from \cite{sa2}.
 More recently, Tirabassi \cite[Theorem 1(i)]{ti} extended Sasaki's theorem to higher syzygies showing that $G^m$ satisfies the property $(N_p)$, if $m \geq p + 2$ (with some constraints on the characteristic of the base field). When $p \geq 1$, she also proved  that taking $m \geq p+1$ suffices, if $\pi^* G = L^2$ and  $L$ has no base divisors \cite[Theorem 2(i)]{ti}.

Note that Theorem \ref{main} already improves the results of Sasaki and Khaled. Indeed, by taking $p=1$,  one gets:
\begin{corollary}
If $m \geq 2$, then $G^m$ satisfies the property $(N_1)$, i.e.,  it is projectively normal \emph{and} the homogeneous ideal 
of $\KA$ in  $\mathbb{P}(H^0(\KA, G^m)^{\vee})$ is generated by quadrics.
\end{corollary}
\noindent Moreover, the bound in Theorem \ref{main} is optimal (at least for $p \leq 1$) 
as, for instance, not any ample line bundle on $\KA$ is projectively normal. 
The projective normality of  ample line bundles on  Kummer varieties has been characterized by Khaled  in \cite[Theorem 2.3]{khaled2}.

The proof of Theorem \ref{main} rests both on standard techniques going back to Mumford and used (at least) by Sekiguchi,  Sasaki and Khaled to study equations defining abelian and Kummer varieties, and on the recent notion of \emph{fractional multiplication maps} of global sections of ample line bundles on abelian varieties,  introduced by Jiang and Pareschi in \cite{jipa} (we recall it, for reader's convenience, in \S \ref{fmm}). What makes them fit together properly is \cite[Proposition 3.5]{ca} and its proof, that will be also recalled in subsection \S \ref{bpfthreshold}.

After this manuscript was finished, the author was informed by Sofia Tirabassi that, jointly with Minyoung Jeon, they also proved Theorem  \ref{main} by a  different but related method. The author thanks her for sharing their preprint  \cite{jeti}.

\section{Kummer varieties: basic properties}\label{kummerbasic}

Let $\K$ be an algebraically closed field of characteristic $\Char(\K) \neq 2$. Given an abelian variety $A$ defined over $\K$, its associated \emph{Kummer variety}  $\KA$ is the quotient of $A$ by the action of the inverse involution 
$\iota \colon A \to A$. 
Let $$\pi \colon A \to \KA$$ be the natural quotient map.
An ample line bundle on $A$ is said to be \emph{totally symmetric}, if it is of the form $\pi^*G$ for a certain ample line bundle $G$ on $\KA$. 
Given an ample line bundle $G$ on $\KA$, we denote by $[-1]$ the canonical automorphism of $H^0(A, \pi^*G)$ induced by $\iota$, and by 
$$H^0(A, \pi^*G)_+ = \pi^* H^0(\KA, G) \subseteq H^0(A, \pi^*G)$$ 
the vector subspace of $H^0(A, \pi^*G)$ consisting of those elements that are invariant under the action of $[ -1 ]$. 

In the rest of this section, mostly following Sasaki \cite{sa1}, we collect some standard Lemmas that will be useful later.

\begin{lemma}\label{lemmads}
$\OO_{\KA}$ is a direct summand of $\pi_*\OO_A$.
\end{lemma}
\begin{proof}
This is true as $\Char(\K) \neq 2 = \deg(\pi)$. See, e.g., \cite[Proposition 5.7(2)]{komo}.
\end{proof}
\begin{lemma}\label{even}
Let $G$ be an ample line bundle on $\KA$. Then there exists an  
ample line bundle $L$ on $A$ such that $\pi^* G = L^2$. 
\end{lemma}
\begin{proof}
This is well-known. A proof can be found in \cite[Proof of Lemma 1.2]{sa1}, where the result 
 follows from \cite[Corollary 4, p.\ 315]{mum1} and \cite[\S 23, Theorem 3]{mum3}. 
\end{proof}

Let us fix some notations.
Given an ample line bundle $L$ on $A$,    there exists an isogeny of abelian varieties associated to  $L$
\[
\phi_L : A \rightarrow \widehat{A}
\]
  which, on 
	 closed points, is given by $a \mapsto t_a^*L \otimes L^{\vee}$, 
where
 $t_a \colon A \to A$ is the translation morphism by $a$ and  
 $\widehat{A} = \Pic0 A$ is the dual abelian variety. 
The kernel of $\phi_L$ is denoted by $K(L)$.
Moreover, if $0 \neq k \in \mathbb Z$,
  $$\mu_k \colon A \to A$$ is  the multiplication-by-$k$ isogeny of $A$ and $A_k := \mathrm{Ker}(\mu_k)$ is its kernel.

\begin{lemma}\label{sasaki1}
Let $G$ be an ample line bundle on $\KA$. Then there exists an isogeny $f \colon A \to A^{\prime}$ and an ample totally symmetric line bundle $(\pi^{\prime})^* G^{\prime}$ on the abelian variety $A^{\prime}$ 
such that $f^* (\pi^{\prime})^* G^{\prime} = \pi^* G$ and $K((\pi^{\prime})^* G^{\prime}) = A^{\prime}_2$
\emph{(here,  $\pi^{\prime} \colon A^{\prime} \to \KApr$ denotes the quotient morphism of $A^{\prime}$)}.
Moreover, 
\begin{equation}\label{G'+}
H^0(A^{\prime}, (\pi^{\prime})^* G^{\prime})_+ = H^0(A^{\prime}, (\pi^{\prime})^* G^{\prime})
\end{equation}
and $G^{\prime}$ is globally generated.
\end{lemma}
\begin{proof}
The existence of $f$ and $G^{\prime}$ with these properties is noted in \cite[Lemma 1.2 and the comment above it]{sa1}. It basically follows as a by-product  of classical results of Mumford  \cite{mum1}. 
In particular, the condition $K((\pi^{\prime})^* G^{\prime}) = A^{\prime}_2$ implies \eqref{G'+} by \cite[\S 3, Inverse Formula, p.\ 331]{mum1}.

We just have to
observe that $G^{\prime}$ is globally generated. This easily follows from \eqref{G'+}. Indeed, 
 by Lemma \ref{even} and the fact that the square of any ample line bundle on an abelian variety is globally generated (see, e.g., \cite[pp.\ 60-61]{mum3}), we know that $(\pi^{\prime})^* G^{\prime}$ is  globally generated, i.e., the evaluation map    
\begin{equation}\label{eqsa1}
H^0(A^{\prime}, (\pi^{\prime})^* G^{\prime}) \otimes \OO_{A^{\prime}} \xrightarrow{\mathrm{ev}} (\pi^{\prime})^* G^{\prime}
\end{equation}
is surjective. 
 If we now apply $\pi^{\prime}_*$ to \eqref{eqsa1},  using  \eqref{G'+} and the projection formula, we get
\[
H^0(\KApr, G^{\prime}) \otimes \pi^{\prime}_*\OO_{A^{\prime}} \to G^{\prime} \otimes \pi^{\prime}_* \OO_{A^{\prime}},
\]
which is surjective too, because $\pi^{\prime}$ is finite.
Since there exists a surjective morphism $\pi^{\prime}_*\OO_{A^{\prime}} \to \OO_{\KApr}$ by Lemma \ref{lemmads}, one has the following commutative diagram 
\[
\xymatrix{
H^0(\KApr, G^{\prime}) \otimes \pi^{\prime}_* \OO_{A^{\prime}}  \ar@{->>}[r] \ar@{->>}[d] &G^{\prime} \otimes \pi^{\prime}_* \OO_{A^{\prime}} \ar@{->>}[d]  \\
H^0(\KApr, G^{\prime}) \otimes \OO_{\KApr}  \ar[r]^-{\mathrm{ev}} &G^{\prime}
}
\]
Hence, $G^{\prime}$ has no base points. 
\end{proof}

\begin{lemma}\label{sasaki2} By keeping the notation of Lemma \ref{sasaki1},
the isogeny $f$   sits in the  commutative diagram
\begin{equation*}\label{maindiag000}
\xymatrix{
A  \ar[r]^-{f} \ar[d]^-{\pi} &A^{\prime}  \ar[d]^-{\pi^{\prime}}  \\
\KA   \ar[r]^{f^{\prime}} &\KApr
}
\end{equation*}
where $f^{\prime}$ is the morphism induced by $f$, and one has
\begin{equation}\label{eqsa2}
(f^{\prime})^* G^{\prime} = G.
\end{equation}
\end{lemma} 
\begin{proof}
This is proved in \cite[p.\ 327]{sa1}.
\end{proof}

\begin{lemma}\label{gg}
Any ample line bundle on $\KA$ is globally generated.
\end{lemma}
\begin{proof}
It follows from \eqref{eqsa2} and the last statement of Lemma \ref{sasaki1}. 
\end{proof}

\section{The Green-Lazarsfeld property $(N_p)$}\label{GLprop}

In this section -- and in the next one -- $\K$ is an algebraically closed field of \emph{arbitrary} characteristic. Let $X$ be a projective variety over $\K$,  and $L$ be an ample line bundle on $X$.
The \emph{section algebra} of $L$ is  $$R_L := \bigoplus_{m \geq 0} H^0(X, L^m).$$ It is a finitely generated  module over the polynomial ring $S_L := \mathrm{Sym} (H^0(X, L))$. Hence, it admits a (unique up to isomorphism) minimal graded free resolution
\[
0 \rightarrow E_d(L) \rightarrow \ldots \rightarrow E_1(L) \rightarrow E_0(L) \rightarrow R_L \rightarrow 0,
\]
and it is natural to ask when the first steps of this resolution are as simple as possible.

\begin{definition}[\cite{grla}]
Let $p$ be an integer $\geq 0$. Then
$L$ is said to \emph{satisfy the property} $(N_p)$, if the first $p$ steps of the minimal graded free resolution of the $S_L$-algebra $R_L$ are linear, i.e., of the form
\[
\xymatrix{
S_L(-(p+1))^{\oplus i_p} \ar[r] \ar@{=}[d] &S_L(-p)^{\oplus i_{p-1}} \ar[r] \ar@{=}[d] &\ldots \ar[r]  &S_L(-2)^{\oplus i_1} \ar[r] \ar@{=}[d] &S_L \ar[r] \ar@{=}[d] &R_L \ar[r]  &0 \\
E_p(L) &E_{p-1}(L) & &E_1(L) &E_0(L)
}\]
\end{definition}

 This has the following geometric interpretation. The property $(N_0)$ means that $L$ is projectively normal. Note that if $L$ is  projectively normal  then $L$  is very ample (see \cite[pp.\ 38-39]{mum2}\footnote{A projectively normal line bundle is called \emph{normally generated} in \cite{mum2}.}). 
If $L$ satisfies $(N_1)$, then
$$\ldots \to S_L(-2)^{\oplus i_1}  \to I_{X/\mathbb{P}} \to 0$$ is
 a resolution of the homogeneous ideal $I_{X/\mathbb{P}}$ of $X \hookrightarrow \mathbb P := \mathbb{P}(H^0(X, L)^{\vee})$. So, in this case,
 $I_{X/\mathbb P}$ is generated by quadrics. The property $(N_2)$ means that the relations among these quadrics
are generated by linear ones, and so on.

\subsection{A  criterion}\label{gencrit} 

Let $M_L$ be the \emph{kernel bundle} of an ample and globally generated line bundle $L$ on $X$, that is the kernel of the evaluation morphism $$0 \to M_L \to H^0(X, L) \otimes \OO_X \xrightarrow{\rm{ev}} L \to 0. $$
If the vanishing
\[
H^1(X, M_L^{\otimes (i+1)} \otimes L^h) = 0
\]
holds for all integers $0 \leq i \leq p$ and
$h \geq 1$, then $L$ satisfies the property $(N_p)$.

\noindent This fact is well-known when $\Char(\K) = 0$ (see, e.g., \cite[pp.\ 510-511]{lasampl}), but it holds true as well in arbitrary characteristic, 
 as proved by the author in \cite[Proposition 4.1]{ca}, 
thanks to
 an algebraic result of Kempf \cite{ke}. We refer the interested reader to \cite[\S 4]{ca} for more details.\footnote{
There is a misprint in the statement of \cite[Proposition 4.1]{ca}, due to the fact that  \cite[Note 5 at p.\ 955]{ca} is slightly imprecise. We take the opportunity to give the correct statement here. The proof in \cite{ca} rests unchanged.
}

\section{Fractional multiplication maps of global sections on abelian varieties}\label{fmm}

Let $A$ be an abelian variety defined over $\K$ (here we still do not make any assumption on the characteristic of the field). 
We first fix some other notations.
A \emph{polarization} $\l \in \mathrm{Pic}\, A / \Pic0 A$ is the class of an ample line bundle $L$ on $A$.
The corresponding isogeny is $$\phi_{\l} = \phi_L \colon A \to \Ad.$$  
Let $\PP$  be the normalized Poincar\'e line bundle on $A \times \Ad$. 
 For any point $\gamma \in \widehat{A}$, the restriction of $\PP$ to $A \times \{\gamma\}$ is denoted by
$P_{\gamma}$.
Given a coherent sheaf $\FF$ on $A$, the (formal) $\Q$\emph{-twisted sheaf} $\FF \la x \l \ra$ is the equivalence class of pairs $(\FF, x \l)$, where $x \in \Q$ and the equivalence relation is
\[
(\FF, (h+x)\l) \sim (\FF \otimes L^{h}, x\l)
\]
for any ample line bundle $L$ representing the polarization $\l$ and $h \in \mathbb{Z}$. 
 To simplify the notation, sometimes we will write $\FF \la x L \ra$ instead of $\FF \la x \l \ra$.
Note that a sheaf $\FF$ may be naturally seen as the $\Q$-twisted sheaf $\FF \la 0 \l \ra$. 
Finally,  the \emph{tensor product} of two $\Q$-twisted sheaves $\FF \la x \l \ra$ and $\GG \la y \l \ra$ on $A$ is
$$\FF \la x \l \ra \otimes \GG \la y \l \ra := (\FF \otimes \GG) \la (x+y) \l \ra,$$ and
 the \emph{pullback} of $\FF \la x \l \ra$ by  an isogeny of abelian varieties $f \colon B \to A$
is 
$$f^*( \FF \la x \l \ra) := (f^* \FF) \la x f^*\l \ra.$$

Let now $L$ be an ample and globally generated line bundle on $A$. 
Jiang and Pareschi \cite{jipa} introduced a ``fractional'' analogue of the multiplication maps of global sections of $L$ (see, in particular, \cite[pp.\ 818-819]{jipa}), and studied it via the Fourier-Mukai-Poincar\'e transform. 

\noindent Given a positive rational number $x = \frac{a}{b}$  
and any closed point $\alpha \in \Pic0 A$, these \emph{fractional maps} are the multiplication maps 
\begin{equation}\label{framul}
H^0(A, L) \otimes H^0(A, L^{ab} \otimes P_{\alpha}) \to H^0(A, \mu_b^* L \otimes L^{ab} \otimes P_{\alpha})
\end{equation}
obtained by composing the natural inclusion $$H^0(A, L) \hookrightarrow H^0(A, \mu_b^*L)$$ with the usual multiplication map of global sections, where $\mu_b \colon A \to A$ is the multiplication by $b$.

\noindent By applying $\mu_b^*$ to the short exact sequence defining the kernel bundle $M_L$ of $L$, and then tensoring it with $L^{ab} \otimes P_{\alpha}$, one gets
\begin{equation}\label{syzbd2}
0 \to \mu_b^*M_L \otimes L^{ab} \otimes P_{\alpha} \to H^0(A, L) \otimes  L^{ab} \otimes P_{\alpha}  \to \mu_b^*L \otimes L^{ab} \otimes P_{\alpha}  \to 0.
\end{equation}
Then, taking the long exact sequence in cohomology, we  see that \eqref{framul} are surjective for any closed point $\alpha \in \Pic0 A$ if (and only if) one has
\begin{equation}\label{survan1}
H^1(A, \mu_b^* M_L \otimes L^{ab} \otimes P_{\alpha}) = 0
\end{equation}
for any closed point $\alpha \in \Pic0 A$.

\begin{definition}
Given   a coherent sheaf   $\FF$ on an abelian variety $A$,  
following Mukai's terminology \cite{mukai},  we say that $\FF$  is  $IT(0)$  (\emph{Index Theorem} with index $0$) 
if $$H^i(A, \FF \otimes P_{\alpha}) = 0,$$ for all $i \geq 1$ and $\alpha \in \Ad(\mathbb K)$.

More generally, a $\Q$-twisted sheaf $\FF \la x \l \ra$ 
 is $IT(0)$ if so is $\mu_b^* \FF \otimes L^{ab}$, where $x = \frac{a}{b}$   
and $L$ is a representative of the polarization $\l$.\footnote{This definition is dictated by the fact that  
$$\mu_b^*(\FF \la \frac{a}{b} \l \ra) = (\mu_b^*\FF) \la \frac{a}{b} \mu_b^*\l \ra = (\mu_b^*\FF) \la \frac{a}{b} \cdot b^2\l \ra = (\mu_b^*\FF) \la ab \l \ra.$$}
\end{definition}

\noindent So the above vanishing \eqref{survan1} holds true for all closed points $\alpha \in \Pic0 A$, if and only if
the $\Q$-twisted sheaf $M_L \la \frac{a}{b} \l \ra$ is $IT(0)$. 
Indeed, for any $i \geq 2$, the vanishing  $$H^i(A, \mu_b^* M_L \otimes L^{ab} \otimes P_{\alpha}) = 0$$   is always satisfied   thanks to \eqref{syzbd2}, and Mumford's vanishing saying that an ample line bundle on an abelian variety has no higher cohomologies, i.e., it is $IT(0)$ (see \cite[\S 16]{mum3}).  

Let us point out that, for any $\Q$-twisted sheaf, the property of being $IT(0)$   does not depend on the chosen representation $x = \frac{a}{b}$, because, if we write $x = \frac{ka}{kb}$, then   
$\FF \la \frac{ka}{kb} \l \ra$ is $IT(0)$ if and only if 
$\mu_k^* (\FF \la \frac{a}{b} \l \ra)$ is $IT(0)$, and this is the case if and only if $\FF \la \frac{a}{b} \l \ra$ is $IT(0)$ (see \cite[Proposition 1.3.3(1)]{catesi} for a characteristic-free proof of a slightly more general property). In particular, the surjectivity of \eqref{framul} for any closed $\alpha \in \Pic0 A$ is independent from the representation $x = \frac{a}{b}$.

\subsection{The basepoint-freeness threshold}\label{bpfthreshold}

Let $a \in A$ be a closed point and $\II_a \subseteq \OO_A$ be its ideal sheaf.  In \cite{jipa}, Jiang and Pareschi defined the \emph{basepoint-freeness threshold}  
\[
\beta(A, \l) = \mathrm{Inf}  \left\{ x \in \Q^+ \ | \ \mathcal{I}_a \la x \l \ra \ \rm{is\   \emph{$IT(0)$}} \right\}
\]
of a polarization $\l$ on $A$.\footnote{The definition given in \cite{jipa} is equivalent to ours thanks to \cite[Lemma 3.3]{ca}, where $\beta(A, \l)$ is denoted  $\epsilon_1(\l)$.}
The name is motivated by the fact that $$\beta(A, \l) \leq 1,$$ and the strict inequality holds if and only if $\l$ is basepoint-free (see \cite[p.\ 842]{jipa}).

When $\l$ is basepoint-free, \cite[Theorem D]{jipa}  gives that the surjectivity of the fractional multiplication maps of global sections of any representative $L$ of the class $\l$, is encoded by the  constant $\beta(A, \l)$. Namely,
\emph{the fractional multiplication maps \eqref{framul} are surjective for all closed points $\alpha \in \Pic0 A$ if and only if}
\begin{equation}\label{thmD}
\beta(A, \l) < \frac{x}{1+x}
\end{equation}
(see also \cite[Remark 2.1]{ca}).
 Using this fact, along with a $\Q$-twisted version of the preservation of vanishing of Pareschi and Popa \cite{papoIII} (that we will recall below), it was proved by the author  in \cite[Proposition 3.5]{ca} that, given an integer $p \geq 0$,  
if $\beta(A, \l) < \frac{1}{p+2}$
 then $M_L^{\otimes (p+1)} \otimes L^h$ is an $IT(0)$ sheaf for all $h \geq 1$ (and hence, by \S \ref{gencrit}, $L$  satisfies the property $(N_p)$).\footnote{See also \cite[Proposition 6.2]{ito} for an improvement of this criterion when $p \geq 1$, which has been recently applied  to the study of higher syzygies of very general abelian surfaces in \cite[Corollary B]{rojas}.}  
Basically, the proof of \cite[Proposition 3.5]{ca} consists in writing the sheaf $M_L^{\otimes (p+1)} \otimes L^h$
 as a product of $IT(0)$ $\Q$-twisted sheaves. This technique  will   also be  useful in the next section.

\section{Proof of  Theorem \ref{main}}\label{main2}

Thanks to the  results we recalled in \S \ref{fmm}, the proof of  Theorem \ref{main} essentially  reduces to apply to certain fractional multiplication maps, a well-established argument   used to study  equations defining abelian  varieties by Sekiguchi \cite{se1}, and later by   Sasaki \cite{sa1} and Khaled \cite{khaled} for Kummers.

We begin giving the following definition, that will be useful in a moment. 
\begin{definition}
 A coherent sheaf $\FF$ on an abelian variety $A$ is said to be a $GV$ (\emph{generic vanishing}) sheaf, if 
\[
\mathrm{codim}_{\Ad} \, V^i(A, \FF) \geq i \quad    \textrm{for all}\ i \geq 0,
\]
 where, for any $i$, 
$V^i(A, \FF)$ is a closed reduced subscheme of $\Ad$ whose set of closed points coincides with
$\{ \alpha \in \Ad(\mathbb K) \ | \ h^i(A, \FF \otimes P_{\alpha}) \neq 0 \}$.

A $\Q$-twisted sheaf $\FF \la x \l \ra$   is $GV$ if so is $\mu_b^* \FF \otimes L^{ab}$, where $x = \frac{a}{b}$ and 
$L$ is an ample line bundle on $A$ representing the polarization $\l$. 
\end{definition}
\begin{remark}\label{gvfac}
Note that an $IT(0)$ ($\Q$-twisted) sheaf is $GV$. Moreover,
if $\FF$ is such that $V^i(A, \FF) = \emptyset$ for $i \geq 2$, then $\FF$ is $GV$ if and only if  $V^1(A, \FF)$ is \emph{properly} contained in $\Pic0 A$, i.e., $h^1(A, \FF \otimes P_{\alpha}) = 0$ for $\alpha \in \Pic0 A$ \emph{general}.
\end{remark}
\noindent For ease of use, we also state the $\Q$-twisted preservation of vanishing, which follows at once from \cite[Proposition 3.1]{papoIII}.    
\begin{proposition}[\cite{ca}, Proposition 3.4]\label{presvan}
Let $\FF_1$ and $\FF_2$ be coherent sheaves on $A$, with one of them locally free. If $\FF_1 \la x \l \ra$ is $IT(0)$ and $\FF_2 \la y \l \ra$ is $GV$, then $\FF_1 \la x \l \ra \otimes \FF_2 \la y \l \ra$ is $IT(0)$.
\end{proposition}

Let us go back to our Kummer variety $\KA$, hence from now on we  assume $\Char(\K) \neq 2$.  
Recall the setting of Lemma \ref{sasaki2}:  given the ample line bundle $G$ on $\KA$, we have a commutative diagram
\begin{equation}\label{maindiag00}
\begin{gathered}
\xymatrix{
A \ar[d]^-{\pi} \ar[r]^-{f} &A^{\prime} \ar[d]^-{\pi^{\prime}}\\
\KA \ar[r]^-{f^{\prime}} &\KApr
}
\end{gathered}
\end{equation}
where $f$ is an isogeny of abelian varieties,  $\pi$ (resp.\ $\pi^{\prime}$) is the quotient morphism of $A$ (resp.\ of $A^{\prime}$),
  $f^{\prime}$ is the morphism induced by $f$, and  there exists  an  ample 
line bundle $G^{\prime}$ on $\KApr$ such that $(f^{\prime})^* G^{\prime} = G$  and $H^0(A^{\prime}, (\pi^{\prime})^* G^{\prime})_+ = H^0(A^{\prime}, (\pi^{\prime})^* G^{\prime})$.

\noindent Let $L^{\prime}$ be an ample  
line bundle on $A^{\prime}$ such that $(\pi^{\prime})^* G^{\prime} = (L^{\prime})^2$ (Lemma \ref{even}).
Since, by Lemma \ref{sasaki1},  
  $\mu_2^{-1}(K(L^{\prime})) = K((L^{\prime})^2) = 
	A^{\prime}_2$,  
	 we have that $K(L^{\prime}) = \mu_2(\mu_2^{-1}(K(L^{\prime}))) =   \{ \OO_{A^{\prime}} \}$, 
  i.e., $h^0(A^{\prime}, L^{\prime}) = 1$. In particular, $L^{\prime}$ is not globally generated and hence
	\[
	\beta(A^{\prime}, (\pi^{\prime})^* G^{\prime}) = \frac{\beta(A^{\prime}, L^{\prime})}{2} = \frac{1}{2},
	\]
  where the first equality follows from the definition of the basepoint-freeness threshold (see \S \ref{bpfthreshold}).
Let $G^m$ be the $m$\emph{-th} power  line bundle on  $\KA$, where $m > \frac{p+2}{2}$ is an integer. As a way  to simplify the notation, let us define 
\[
H := G^m \quad \textrm{and} \quad H^{\prime} := (G^{\prime})^m.
\]
 Then,  
$(f^{\prime})^* H^{\prime} = H$ and $$\beta(A^{\prime}, (\pi^{\prime})^* H^{\prime}) = \frac{1}{2m} < \frac{1}{p+2}.$$ 
By Lemma \ref{gg},  $H$ is globally generated. So we have the short exact sequence
\begin{equation}\label{mainpr1}
0 \to M_H \to H^0(\KA, H) \otimes \OO_{\KA} \to H \to 0,
\end{equation}
where $M_H$ is the kernel bundle associated to $H$.
By \S \ref{gencrit}, in order to prove that $H$ satisfies the property $(N_p)$, we need to show that
\begin{equation}\label{genvan}
H^1(\KA, M_H^{\otimes (p+1)} \otimes H^h) = 0
\end{equation}
for all $h \geq 1$.
Thanks to Lemma \ref{lemmads} and the projection formula,
  one has    
	\begin{equation*}
	\begin{split}
	H^1(\KA, M_H^{\otimes (p+1)} \otimes H^h) &\subseteq   H^1(\KA, M_H^{\otimes (p+1)} \otimes H^h \otimes \pi_*\OO_A) \\
	&= H^1(A, \pi^*(M_H^{\otimes (p+1)} \otimes H^h)).
	\end{split}
	\end{equation*}
 Therefore, \eqref{genvan} holds true for all $h \geq 1$, if $H^1(A, \pi^*(M_H^{\otimes (p+1)} \otimes H^h)) = 0$ for all $h \geq 1$.
Using the $\Q$-twisted notation, let us write 
\[
\pi^*(M_H^{\otimes (p+1)} \otimes H^h) 
 = 
\big(\pi^*M_H \la \, \frac{1}{p+1} \, \pi^* H\, \ra \big)^{\otimes (p+1)} \otimes \pi^*H^{(h-1)}.
\]
Note that the line bundle  $\pi^*H^{(h-1)}$ is either ample, and hence $IT(0)$, or it is trivial.
Then, to get Theorem \ref{main} , thanks to Proposition \ref{presvan} it suffices to prove that
\begin{lemma}\label{IT0pas}
 The $\Q$-twisted sheaf $\pi^*M_H \la \, \frac{1}{p+1} \, \pi^* H\, \ra$ is $IT(0)$.
\end{lemma}
\noindent We will give the proof of this fact in the remaining part of the paper.
\begin{proof}[Proof of Lemma \ref{IT0pas}]
Keeping in mind  \eqref{maindiag00}, 
 let $x_0 < \frac{1}{p+1}$ be a positive rational number such that 
\[
\beta(A^{\prime}, (\pi^{\prime})^*H^{\prime}) < \frac{x_0}{1+x_0} < \frac{1}{p+2}. 
\]
Write $x_0 = \frac{a}{b}$.
We first aim to prove that 
\begin{equation}\label{aim1}
H^1(A, \mu_{b}^* \pi^*M_H \otimes \pi^*H^{ab}) = 0.
\end{equation}
Note that,  pulling back 
\eqref{mainpr1} by $\pi$, one has
\begin{equation}\label{maindiag1}
0  \to \pi^*M_H \to H^0(A, \pi^*H)_+ \otimes \OO_A  \to \pi^*H \to 0. 
\end{equation}
Then, taking the pullback 
 by $\mu_{b}$ of \eqref{maindiag1} and tensoring it with $\pi^*H^{ab}$,  we get
\[
0  \to \mu_{b}^* \pi^*M_H \otimes \pi^*H^{ab}  \to H^0(A, \pi^*H)_+ \otimes \pi^*H^{ab}  \to \mu_{b}^* \pi^*H \otimes \pi^* H^{ab}  \to 0. 
\]
So, in order to obtain  the desired vanishing above, it is enough to prove the surjectivity of  the  map
\begin{equation}\label{maintes}
 H^0(A, \pi^*H)_+ \otimes H^0 (A, \pi^*H^{ab}) \to H^0(A, \mu_{b}^* \pi^*H \otimes \pi^* H^{ab})
\end{equation}
obtained by composing the  inclusion $H^0(A, \pi^*H)_+  \hookrightarrow H^0(A, \mu_{b}^* \pi^*H)$ with the usual multiplication map of global sections. Since the next diagram  is commutative,
 it suffices in turn to show that its horizontal map  is surjective:
\begin{equation}\label{surhor}
\begin{gathered}
\xymatrix{
\big( H^0(A, \pi^* G)_+ \big)^{\otimes m} \otimes H^0(A, \pi^* H^{ab}) \ar[r] \ar[d] &H^0(A, \mu_b^* \pi^* H \otimes \pi^* H^{ab}) \\
H^0(A, \pi^*H)_+ \otimes H^0 (A, \pi^*H^{ab}) \ar[ur]_-{\eqref{maintes}}
}
\end{gathered}
\end{equation}
The horizontal map in \eqref{surhor} is of course 
 defined similarly to \eqref{maintes}: namely, we compose the natural inclusion $$\big( H^0(A, \pi^* G)_+ \big)^{\otimes m} 
 \hookrightarrow \big( H^0(A, \mu_b^* \pi^* G) \big)^{\otimes m}$$ with the  multiplication map of global sections
$$
\big( H^0(A, \mu_b^* \pi^* G) \big)^{\otimes m} \otimes H^0(A, \pi^* H^{ab}) \rightarrow H^0(A, \mu_b^* \pi^* H \otimes \pi^* H^{ab}).
$$

With the goal of proving the  surjectivity of the horizontal map in \eqref{surhor},
 let us first consider the following  diagram of maps of global sections on $A^{\prime}$, where   $\beta_0 \in \Pic0 A^{\prime}$ is \emph{general}, and
$\delta$ is an arbitrarily fixed closed point of $\Pic0 A^{\prime}$.
\begin{equation}\label{m2}
\begin{gathered}
\resizebox{6.0 in}{!}{
\xymatrix{
\big(H^0((\pi^{\prime})^* G^{\prime}) \big)^{\otimes (m-1)} \otimes H^0((\pi^{\prime})^* G^{\prime} \otimes P_{\beta_0}) 
\otimes H^0((\pi^{\prime})^* (H^{\prime})^{ab} \otimes P_{\delta}) \ar[r] \ar@{->>}[d]
 &H^0(\mu_b^* (\pi^{\prime})^* H^{\prime} \otimes (\pi^{\prime})^* (H^{\prime})^{ab} \otimes P_{b\beta_0 + \delta}) \\
H^0((\pi^{\prime})^* H^{\prime} \otimes P_{\beta_0}) \otimes H^0((\pi^{\prime})^* (H^{\prime})^{ab} \otimes P_{\delta}) \ar@{->>}[ur]
}}
\end{gathered}
\end{equation}
The vertical arrow is given by usual multiplication maps of global sections. It is surjective by classical results of Sekiguchi \cite[Main theorem]{se2} (see also \cite[Remark below Proposition 1.5]{se1}) and \cite[Theorem 2.4]{se3}. 
The diagonal arrow is a
 fractional multiplication map,   
which is surjective  for any closed point $\delta \in \Pic0 A^{\prime}$. Indeed, $\beta(A^{\prime}, (\pi^{\prime})^* H^{\prime}) < \frac{x_0}{1+ x_0}$ and hence the $\Q$-twisted sheaf $M_{(\pi^{\prime})^*H^{\prime}} \la \frac{a}{b} (\pi^{\prime})^* H^{\prime} \ra$ is $IT(0)$ (see \S \ref{fmm} and, in particular, \eqref{thmD}). If $z_0 \in A^{\prime}$ is such that $t_{z_0}^* (\pi^{\prime})^* H^{\prime} = (\pi^{\prime})^* H^{\prime} \otimes P_{\beta_0}$, then 
\begin{equation*}
\begin{split}
M_{(\pi^{\prime})^*H^{\prime} \otimes P_{\beta_0}} \la \frac{a}{b} (\pi^{\prime})^* H^{\prime} \ra &= t_{z_0}^* (M_{(\pi^{\prime})^*H^{\prime}}) \la \frac{a}{b} (\pi^{\prime})^* H^{\prime} \ra \\
 &= t_{z_0}^* (M_{(\pi^{\prime})^*H^{\prime}} \la \frac{a}{b} (\pi^{\prime})^* H^{\prime} \ra). 
\end{split}
\end{equation*}
Therefore,  $M_{(\pi^{\prime})^*H^{\prime} \otimes P_{\beta_0}} \la \frac{a}{b} (\pi^{\prime})^* H^{\prime} \ra$ is $IT(0)$ as well by \cite[Proposition 1.3.3(1)]{catesi}, i.e., the diagonal arrow in \eqref{m2} is surjective.

Now, since the maps in \eqref{m2} are surjective for a general $\beta_0$ and since the set of torsion points is dense in $A^{\prime}$, we can assume that $\beta_0$ is a $k b$-torsion point for a certain $k \geq 1$, and, since we are free to write 
 $\frac{a}{b} = \frac{ka}{kb}$, we can actually assume 
$k = 1$.
So there exists a $b$-torsion closed point, that for simplicity we still denote by $\beta_0$, such that the composition \eqref{m2} simplifies as
\begin{multline}\label{m2s}
\big(H^0((\pi^{\prime})^* G^{\prime}) \big)^{\otimes (m-1)} \otimes H^0((\pi^{\prime})^* G^{\prime} \otimes P_{\beta_0}) 
\otimes H^0((\pi^{\prime})^* (H^{\prime})^{ab} \otimes P_{\delta}) 	\\
	\twoheadrightarrow
 H^0(\mu_b^* (\pi^{\prime})^* H^{\prime} \otimes (\pi^{\prime})^* (H^{\prime})^{ab} \otimes P_{\delta}).
\end{multline}
Let $z_1 \in A^{\prime}$ be a closed point  such that $t_{z_1}^* (\pi^{\prime})^* G^{\prime} = (\pi^{\prime})^* G^{\prime} \otimes P_{\beta_0}$, and let $z_2 \in A^{\prime}$ with $b z_2 = z_1$. Then, 
\begin{equation}\label{z2}
t_{z_2}^*\mu_b^* (\pi^{\prime})^* G^{\prime} = \mu_b^* t_{z_1}^* (\pi^{\prime})^* G^{\prime} = \mu_b^* (\pi^{\prime})^* G^{\prime}.
\end{equation}
Moreover, given any closed point  $\beta \in \Pic0 A^{\prime}$, one has that
\begin{equation}\label{z22}
t_{z_2}^*((\pi^{\prime})^*(H^{\prime})^{ab} \otimes P_{\beta}) = (\pi^{\prime})^*(H^{\prime})^{ab} \otimes P_{\delta}
\end{equation}
 for a certain $\delta \in \Pic0 A^{\prime}$.\footnote{Actually, it is easy to see that $\delta = \beta$. This precision is however unnecessary here.} 
Therefore, from \eqref{z2} and \eqref{z22}
we get a commutative diagram
\begin{equation*}
\begin{gathered}
\resizebox{6.0 in}{!}{
\xymatrix{
\big(H^0((\pi^{\prime})^* G^{\prime})\big)^{\otimes (m-1)} \otimes 
H^0((\pi^{\prime})^* G^{\prime} \otimes P_{\beta_0}) \otimes H^0((\pi^{\prime})^* (H^{\prime})^{ab} \otimes P_{\delta}) \ar[r] &H^0(\mu_b^*(\pi^{\prime})^* H^{\prime} \otimes (\pi^{\prime})^* (H^{\prime})^{ab} \otimes P_{\delta}) \\
\big(H^0((\pi^{\prime})^* G^{\prime})\big)^{\otimes (m-1)} \otimes H^0((\pi^{\prime})^* G^{\prime}) \otimes 
H^0((\pi^{\prime})^* (H^{\prime})^{ab} \otimes P_{\beta}) \ar[u]^-{\mathrm{id}\, \otimes\, t_{z_1}^* \, \otimes \, t_{z_2}^*} \ar[r] &H^0(\mu_b^*(\pi^{\prime})^* H^{\prime} \otimes (\pi^{\prime})^* (H^{\prime})^{ab} \otimes P_{\beta}) \ar[u]^-{t_{z_2}^*} 
}}
\end{gathered}
\end{equation*}
where the top arrow is \eqref{m2s}, which is surjective, and the bottom one is  similarly defined.

To sum up, by the above discussion 
 we  have obtained that
the maps of global sections on $A^{\prime}$ 
\begin{multline}\label{fracprime}
\big( H^0(A^{\prime}, (\pi^{\prime})^* G^{\prime}) \big)^{\otimes m} \otimes H^0 (A^{\prime}, (\pi^{\prime})^*(H^{\prime})^{ab} \otimes P_{\beta})  \\
\to H^0(A^{\prime}, \mu_{b}^* (\pi^{\prime})^*H^{\prime} \otimes (\pi^{\prime})^* (H^{\prime})^{ab} \otimes P_{\beta})
\end{multline}
are surjective, for \emph{all} closed points  $\beta \in \Pic0 A^{\prime}$.
Using this fact, we will prove ahead the following
\begin{claim}\label{claim0}
Let $f^*\big(H^0(A^{\prime}, (\pi^{\prime})^*G^{\prime})\big)$ be the image of the map $$f^* \colon H^0(A^{\prime}, (\pi^{\prime})^*G^{\prime})  \to H^0(A, \pi^*G).$$ 
Then the natural map
\begin{equation}\label{eqclaim}
f^*\big(H^0(A^{\prime}, (\pi^{\prime})^*G^{\prime})\big)^{\otimes m} \otimes H^0 (A, \pi^*{H}^{ab}) \to H^0(A, \mu_{b}^* \pi^*H \otimes \pi^* H^{ab}),
\end{equation}
given by composing the inclusion 
\begin{equation*}
f^*\big(H^0(A^{\prime}, (\pi^{\prime})^*G^{\prime})\big)^{\otimes m} \hookrightarrow \big( H^0(A, \pi^*G) \big)^{\otimes m} \hookrightarrow \big( H^0(A, \mu_{b}^* \pi^*G) \big)^{\otimes m}
\end{equation*}
 with the usual multiplication map of global sections,
  is surjective.  
\end{claim}

\noindent Assuming the  Claim for the moment,  
the surjectivity of the horizontal map in \eqref{surhor} finally follows  because 
\[
f^*\big(H^0(A^{\prime}, (\pi^{\prime})^*G^{\prime})\big)^{\otimes m} = 
  f^*  \big( H^0(A^{\prime}, (\pi^{\prime})^* G^{\prime})_+\big)^{\otimes m} 
 \subseteq  \big( H^0(A, \pi^*G)_+ \big)^{\otimes m},
\]
where the 
equality is \eqref{G'+}.

So far,  we only proved \eqref{aim1}. Then, 
by upper-semicontinuity, 
 $$H^1(A, \mu_{b}^* \pi^*M_H \otimes \pi^*H^{ab} \otimes P_{\alpha}) = 0$$ for $\alpha \in \Pic0 A$ general.
This means (see Remark \ref{gvfac} above) that the $\Q$-twisted sheaf $\pi^*M_H \la \, x_0\,  \pi^* H\, \ra$ is $GV$. But
\[
\pi^*M_H \la \, \frac{1}{p+1} \, \pi^* H\, \ra = \pi^*M_H \la \, x_0\,  \pi^* H\, \ra \otimes \OO_A \la \, (\frac{1}{p+1} - x_0) \, \pi^* H\, \ra,
\]
and $\OO_A \la \, (\frac{1}{p+1} - x_0) \, \pi^* H\, \ra$ is $IT(0)$, as $\frac{1}{p+1} - x_0 > 0$. Therefore,  $\pi^*M_H \la \, \frac{1}{p+1} \, \pi^* H\, \ra$ is $IT(0)$ by  applying Proposition \ref{presvan} once again.
\end{proof}

It only remains to give the 
\begin{proof}[Proof of  Claim \ref{claim0}]
This is an application of the theta structure theorem of Mumford (see \cite[Appendix]{se3} for a characteristic-free discussion), that is the irreducibility of the Heisenberg representation of the theta group. 
The argument below goes similarly to \cite[Proof of Proposition 1.1]{sa1},  replacing with the surjectivity of \eqref{fracprime} 
  Sekiguchi's result \cite[Theorem 2.4]{se3} 
 on the surjectivity of usual multiplication maps  of global sections of line bundles on abelian varieties,
which extends to arbitrary characteristic  
a  previous theorem of Koizumi \cite{koizumi} for complex abelian varieties.

In order (to try) to simplify the notation, we put
\[\LL := \mu_b^*\pi^*H \otimes \pi^*H^{ab}
\]
and
\[
 \MM := \mu_b^*(\pi^{\prime})^*H^{\prime} \otimes (\pi^{\prime})^*(H^{\prime})^{ab}
\]
 and note that $f^*\MM = \LL$.  
As a consequence of Mumford's theorem (see \cite[Corollary]{se2}), given two vector subspaces $V$ and $W$ of $H^0(A, \LL)$ with $V \supseteq W \neq 0$, if  

\noindent 
$\bullet$ for any local ring $(R, \m)$  over $\K$ with residue field $\K$ and any $R$-valued point $\lambda$ of the theta group $\mathcal{G}(\LL)$ one has $U_{\lambda}(W \otimes R) \subseteq V \otimes R$,

\noindent then $V = H^0(A, \LL)$. 

\noindent Let us recall that  $\mathcal{G}(\LL)(R)$ is the set of
 pairs $(a, \varphi)$ where $a \in A(R)$  and $\varphi \colon t_a^* \LL_S \xrightarrow{\simeq} \LL_S$ is an isomorphism.
Here,  $\LL_S := p^* \LL$ is the pullback by the first projection $p \colon A \times S \to A$, where  $S := \mathrm{Spec}(R)$.  
The theta group $\mathcal{G}(\LL)$  is a group scheme naturally acting  on $H^0(A, \LL)$, and
its action is denoted by $U$.

 We want to apply the above criterion with $V$ equal to the  image of the map \eqref{eqclaim}, and $W := f^*(H^0(A^{\prime}, \MM))$
(note that the inclusion $W \subseteq V$ follows from the surjectivity of \eqref{fracprime} with $\beta = \OO_{A^{\prime}}$). 
Let   $K(\LL)$ be the kernel of the isogeny $\phi_{\LL} \colon A \to \Ad$ defined by $\LL$. There is, by definition, a  canonical surjection
$j \colon \mathcal{G}(\LL) \to K(\LL)$. Let
$u := j(\lambda)  \in K(\LL)(R)$. 
We have the following commutative diagram 
\begin{equation}\label{comd0}
\begin{gathered}
\resizebox{6.0 in}{!}{
\xymatrix{
H^0(A_S, \LL_S) \simeq H^0(A, \LL) \otimes R \ar[r]^-{t_u^*} \ar@/^2pc/[rr]^-{U_{\lambda}} &H^0(A_S, t_u^*(\LL_S)) \ar[r]^-{\simeq}
 &H^0(A_S, \LL_S) \\
H^0(A^{\prime}_S, \MM_S) \simeq H^0(A^{\prime}, \MM) \otimes R \ar[u]^-{f^*} \ar[r]^-{t_{f(u)}^*} & H^0(A^{\prime}_S, t_{f(u)}^*\MM_S) \simeq H^0(A^{\prime}_S, \MM_S \otimes  P_{\gamma})  \ar[u]^-{f^*}
}}
\end{gathered}
\end{equation}
where  $A_S := A \times S$, $A^{\prime}_S := A^{\prime} \times S$, and 
$\gamma := \phi_{\MM}(f(u)) \in \widehat{A^{\prime}}(R)$.
On the other hand,
 since $f^*(P_{\gamma}) = \phi_{f^*\MM}(u) = \phi_{\LL}(u) = \OO_A$,
 we also get the  commutative diagram 
\begin{equation}\label{comd2}
\begin{gathered}
\resizebox{5.7 in}{!}{
\xymatrix{
(f^*\big(H^0(A^{\prime}, (\pi^{\prime})^*G^{\prime})\big)^{\otimes m} \otimes R) \otimes (H^0 (A, \pi^*{H}^{ab}) \otimes R) \ar[r]^-{\eqref{eqclaim} \otimes R} &H^0(A_S, \LL_S)  \\
( \big(H^0(A^{\prime}, (\pi^{\prime})^*G^{\prime})\big)^{\otimes m} \otimes R)\otimes H^0 (A^{\prime}_S, ((\pi^{\prime})^*(H^{\prime})^{ab})_S \otimes P_{\gamma}) \ar[r] \ar[u]^-{(f^*)^{\otimes m} \otimes f^*}  &H^0(A^{\prime}_S, \MM_S \otimes P_{\gamma}) \ar[u]^-{f^*} 
}}
\end{gathered}
\end{equation}
where the bottom arrow is defined  by taking the composition  
\begin{equation*}
\begin{gathered}
\resizebox{5.7 in}{!}{
\xymatrix{
\big(H^0(A^{\prime}, (\pi^{\prime})^*G^{\prime})\big)^{\otimes m} \otimes R\, 
\ar@{^{(}->}[r] \ar[rd] &\big(H^0(A^{\prime}, \mu_b^* (\pi^{\prime})^*G^{\prime})\big)^{\otimes m} \otimes R  \ar[d] \\
&H^0(A^{\prime}, \mu_b^* (\pi^{\prime})^*H^{\prime}) \otimes R \simeq H^0(A^{\prime}_S, (\mu_b^* (\pi^{\prime})^*H^{\prime})_S)
}}
\end{gathered}
\end{equation*}
 and then in turn composing  it with the usual multiplication map of global sections
\[
H^0(A^{\prime}_S, (\mu_b^* (\pi^{\prime})^*H^{\prime})_S) \otimes H^0 (A^{\prime}_S, ((\pi^{\prime})^*(H^{\prime})^{ab})_S \otimes P_{\gamma}) \to H^0(A^{\prime}_S, \MM_S \otimes P_{\gamma}).
\]
From the commutativity of the two  diagrams \eqref{comd0} and \eqref{comd2},  we get that $\bullet$ above  holds true if
 the bottom arrow in \eqref{comd2} is surjective.

To check this last fact, 
 consider the following commutative diagram 
\begin{equation}\label{comd1}
\begin{gathered}
\resizebox{5.7 in}{!}{
\xymatrix{
( \big(H^0(A^{\prime}, (\pi^{\prime})^*G^{\prime})\big)^{\otimes m} \otimes R) \otimes H^0 (A^{\prime}_S, ((\pi^{\prime})^*(H^{\prime})^{ab})_S \otimes P_{\gamma}) \ar[r] \ar[d]  &H^0(A^{\prime}_S, \MM_S \otimes P_{\gamma}) \ar[d] \\
 \big(H^0(A^{\prime}, (\pi^{\prime})^*G^{\prime})\big)^{\otimes m} \otimes H^0 (A^{\prime}, (\pi^{\prime})^*(H^{\prime})^{ab} \otimes P_{\bar{\gamma}}) \ar[r]
 &H^0(A^{\prime}, \MM \otimes P_{\bar{\gamma}})
}}
\end{gathered}
\end{equation}
where the vertical arrows are reductions modulo the maximal ideal $\m$ and  $$\bar{\gamma} \colon \mathrm{Spec}(R/\m) \to \mathrm{Spec}(R) \xrightarrow{\gamma} \widehat{A^{\prime}}.$$ 
The bottom arrow in \eqref{comd1} is one of the map considered in \eqref{fracprime} and, hence,
it is surjective. So, by Nakayama's lemma, the top arrow in \eqref{comd1} (which equals the bottom arrow in \eqref{comd2}) is  surjective, too.  
\end{proof}

\section*{Acknowledgment}
I thank Atsushi Ito, Zhi Jiang, Luigi Lombardi and Andr\'es Rojas 
for their useful comments.  
I am extremely grateful, in particular, to Ito for having  pointing out some  errors in a first version of this work.
Finally, I thank the referee for carefully reading the paper and for helpful suggestions.

\providecommand{\bysame}{\leavevmode\hbox
to3em{\hrulefill}\thinspace}


\begin{thebibliography}{EMS}


\bibitem[BL]{bala} P. Bangere and J. Lacini, {Syzygies of adjoint linear series on projective varieties}, arXiv:2302.02517 (2023).


\bibitem[Ca1]{catesi} F. Caucci, {The basepoint-freeness threshold, derived invariants of irregular varieties, and stability of syzygy bundles}, Ph.D. Thesis Sapienza Universit\`a di Roma, 2020. Available at 
 \href{https://iris.uniroma1.it/handle/11573/1355436}{https://iris.uniroma1.it/handle/11573/1355436}.

\bibitem[Ca2]{ca} F. Caucci, {The basepoint-freeness threshold and syzygies of abelian varieties}, Algebra Number Theory \textbf{14} (2020), no. 4, 947--960.





\bibitem[EL]{el} L. Ein and R. Lazarsfeld, {Syzygies and Koszul cohomology of smooth projective varieties of arbitrary dimension}, Invent. 
Math. \textbf{111} (1993), no. 1, 51--67.








\bibitem[Gr]{green} M. Green, {Koszul cohomology and the geometry of projective varieties}, J. Differ. Geom. \textbf{19} (1984), no. 1, 125--171.



\bibitem[GL]{grla} M. Green and R. Lazarsfeld, {On the projective normality of complete linear series on an algebraic curve}, Invent. Math., \textbf{83} (1986), no. 1, 73--90.








\bibitem[It]{ito} A. Ito, {$M$-regularity of $\Q$-twisted sheaves and its application to linear systems on abelian varieties}, Trans. Am. Math. Soc. \textbf{375}, No. 9, 6653--6673 (2022).

\bibitem[JT]{jeti} M. Jeon and S. Tirabassi, {Cohomological rank functions and syzygies of Kummer varieties}, arXiv:2303.16023 (2023). 




\bibitem[JP]{jipa} Z. Jiang and G. Pareschi, {Cohomological rank functions on abelian varieties}, 
Ann. Sci. \'Ec. Norm. Sup\'er. (4) \textbf{53} (2020), no. 4, 815--846.


\bibitem[Ji]{zhinew} Z. Jiang, {Syzygies of some rational homogeneous varieties}, arXiv:2108.02710 (2021). 

\bibitem[Kh1]{khaled}  A. Khaled, 
{\'Equations des vari\'et\'es de Kummer},
Math. Ann. \textbf{295} (1993), no. 4, 685--701.

\bibitem[Kh2]{khaled2} A. Khaled, 
{Projective normality and equations of Kummer varieties},
J. Reine Angew. Math. \textbf{465} (1995), 197--217.

\bibitem[Ke]{ke} G. Kempf, {Projective coordinate rings of abelian varieties}, in {\em Algebraic analysis, geometry and number theory}, Johns Hopkins Univ. Press, Baltimore, MD, 1989, 225--235.


\bibitem[Ko]{koizumi} S. Koizumi, {Theta relations and projective normality of Abelian varieties}, Amer. J. Math., \textbf{98} (1976), no. 4, 865--889.



\bibitem[KM]{komo} J. Koll\'{a}r and S. Mori, {\em Birational geometry of algebraic varieties}, Cambridge University Press, Cambridge, 1998.



\bibitem[La1]{lasampl} R. Lazarsfeld,  {A sampling of vector bundle techniques in the study of linear series}, in {\em Lectures on Riemann surfaces}, World Sci. Publ., Teaneck, NJ, 1989, 500--559.



\bibitem[La2]{laI} R. Lazarsfeld, {\em Positivity in algebraic geometry I}, Springer-Verlag, Berlin, 2004.










\bibitem[Mu]{mukai} S. Mukai, {Duality between $D(X)$ and $D(\widehat{X})$ with its application to Picard sheaves}, Nagoya Math. J. \textbf{81} (1981), 153--175.

\bibitem[M1]{mum1} D. Mumford, {On the equations defining abelian varieties I}, Invent. Math. \textbf{1} (1966), 287--354.


\bibitem[M2]{mum2} D. Mumford, {Varieties defined by quadratic equations}, in {\em Questions on Algebraic Varieties (C.I.M.E., III Ciclo, Varenna, 1969)}, Edizioni Cremonese, Rome, 1970, 29--100. 

\bibitem[M3]{mum3} D. Mumford, {\em Abelian varieties}, Second edition, Oxford University Press, Oxford, 1974.

\bibitem[Pa]{pa1} G. Pareschi,  {Syzygies of abelian varieties}, J. Amer. Math. Soc. \textbf{13} (2000), no. 3, 651--664.






\bibitem[PP]{papoIII} G. Pareschi and M. Popa, {Regularity on abelian varieties III: relationship with generic vanishing and applications}, in {\em Grassmannians, moduli spaces and vector bundles}, Clay Math. Proc., \textbf{14}, Amer. Math. Soc., Providence, RI, 2011, 141--167.


\bibitem[Ro]{rojas} A. Rojas, 
{The basepoint-freeness threshold of a very general abelian surface},
Selecta Math. (N.S.) \textbf{28} (2022), no. 2, Paper No. 34, 14 pp.


 
\bibitem[Sa1]{sa1} R. Sasaki,  {Bounds on the degree of the equations defining Kummer varieties},
J. Math. Soc. Japan \textbf{33} (1981), no. 2, 323--333.

\bibitem[Sa2]{sa2} R. Sasaki, 
{On the equations defining Kummer varieties},
J. Math. Soc. Japan \textbf{34} (1982), no. 2, 223--239.


\bibitem[Se1]{se3} T. Sekiguchi,  {On projective normality of Abelian varieties II}, J. Math. Soc. Japan \textbf{29} (1977), no. 4, 709--727.


\bibitem[Se2]{se1} T. Sekiguchi, 
{On the cubics defining abelian varieties},
J. Math. Soc. Japan \textbf{30} (1978), no. 4, 703--721.


\bibitem[Se3]{se2} T. Sekiguchi, 
{On the normal generation by a line bundle on an Abelian variety},
Proc. Japan Acad. Ser. A Math. Sci. \textbf{54} (1978), no. 7, 185--188.






\bibitem[Ti1]{tithesis} S. Tirabassi, {Syzygies, Pluricanonical Maps, and the Birational Geometry of Varieties of Maximal Albanese Dimension}, Ph.D. Thesis,  Universit\`a degli Studi ``Roma Tre'', 2011, arXiv:1210.0324. 

\bibitem[Ti2]{ti} S. Tirabassi, 
{Syzygies and equations of Kummer varieties}, Bull. Lond. Math. Soc. \textbf{45} (2013), no. 3, 651--665.



\end{thebibliography}
\end{document}